\documentclass[10pt, a4paper]{amsart}
\usepackage{amssymb, amsmath, amsfonts, mathrsfs, amsthm}
\usepackage[utf8]{inputenc}
\usepackage{lmodern}
\usepackage[T1]{fontenc}
\usepackage[a4paper,includeheadfoot,margin=2.54cm]{geometry}

\DeclareMathOperator{\dival}{\mathrm{DivVal}}
\DeclareMathOperator{\Vol}{\mathrm{Vol}}
\DeclareMathOperator{\Val}{\mathrm{Val}}

\DeclareMathOperator{\aut}{\mathrm{Aut}}

\DeclareMathOperator{\XX}{\mathcal{X}}
\DeclareMathOperator{\LL}{\mathcal{L}}
\DeclareMathOperator{\II}{\mathscr{I}}
\DeclareMathOperator{\OO}{\mathscr{O}}
\DeclareMathOperator{\PP}{\mathbb{P}}
\DeclareMathOperator{\ord}{\mathrm{ord}}

\DeclareMathOperator{\QQ}{\mathbb{Q}}
\DeclareMathOperator{\Qlin}{\sim_{\mathbb{Q}}}
\DeclareMathOperator{\can}{\mathrm{can}}
\DeclareMathOperator{\Pic}{\mathrm{Pic}}
\DeclareMathOperator{\lct}{\mathrm{lct}}

\begin{document} 

\title{Delta-invariants for Fano varieties with large automorphism groups}
\date{}
\author{Aleksei Golota}
\thanks{The author is partially supported by Laboratory of Mirror Symmetry NRU HSE, RF Government grant, ag. 14.641.31.0001}

\theoremstyle{definition}

\newtheorem{thm}{Theorem}[section]
\newtheorem{defi}[thm]{Definition}
\newtheorem{prop}[thm]{Proposition}
\newtheorem{rmk}[thm]{Remark}
\newtheorem{cor}[thm]{Corollary}
\newtheorem{exa}[thm]{Example}

\begin{abstract} For a variety $X$, a big $\mathbb{Q}$-divisor $L$ and a closed connected subgroup $G \subset \aut(X, L)$ we define a $G$-invariant version of the $\delta$-threshold. We prove that for a Fano variety $(X, -K_X)$ and a connected subgroup $G \subset \aut(X)$ this invariant characterizes $G$-equivariant uniform $K$-stability. We also use this invariant to investigate $G$-equivariant $K$-stability of some Fano varieties with large groups of symmetries, including spherical Fano varieties. We also consider the case of $G$ being a finite group.
\end{abstract}

\maketitle

\section{Introduction} The problem of constructing K\"ahler--Einstein metrics on Fano varieties (over the field $\mathbb{C}$ of complex numbers) has been extensively studied in recent years. In particular, for smooth Fano varieties (Fano manifolds) the existence of K\"ahler--Einstein metrics was shown by Chen, Donaldson and Sun \cite{CDS15} and Tian \cite{Ti15} to be equivalent to an algebro-geometric condition of {\em $K$-polystability}. Another approach to this problem is the variational one, developed in \cite{BBGZ13, Ber16, BBEGZ19}. For a Fano variety with finite automorphism group the existence of a K\"ahler--Einstein metric is equivalent to a stronger property of {\em uniform $K$-stability}. This was shown in \cite{BBJ18} for $X$ smooth and in \cite{Ber16, LTW19} for a Fano variety with klt singularities.

In view of the above results, it is important to be able to check if a given Fano variety $X$ is $K$-polystable or uniformly $K$-stable. A priori this requires computing certain numerical invariants for all polarized one-parameter degenerations of $X$ (see Definition \ref{defi:Kstab} below for the precise definition of $K$-stability). Ideally, one would like to have a numerical invariant, depending on the variety $X$ and an ample (or, more generally, big) polarization $L$ only, such that the $K$-stability of $(X, L)$ is detected by this invariant. The first example of such invariant was the $\alpha$-invariant (or its version $\alpha_G(X)$ for a compact group $G$ of symmetries of $X$) introduced by Tian \cite[p. 229]{Ti87} via analytic methods. Tian gave a sufficient condition for the existence of a K\"ahler--Einstein metric on a Fano manifold in terms of $\alpha_G(X)$. 

\begin{thm}[{\cite[Theorems 2.1 and 4.1]{Ti87}}] \label{thm:Tian} Let $X$ be a Fano manifold of dimension $n$ and $G \subset \aut(X)$ a compact subgroup. If $\alpha_G(X) > \frac{n}{n+1}$ then $X$ admits a K\"ahler--Einstein metric.
\end{thm}

In early 2000s it became evident to experts that Tian's $\alpha$-invariant coincides with the global log canonical threshold of $X$ (see \cite{C08} and \cite[Appendix A]{CS08}). An algebraic counterpart of Tian's result was given by Odaka and Sano \cite[Theorem 1.4]{OS12}. They have shown by purely algebraic methods that a klt Fano variety $X$ satisfying $\alpha(X) > n/(n+1)$ is $K$-stable. Moreover, for a Fano variety $X$ and a closed subgroup $G \subset \aut(X)$ Odaka and Sano proved \cite[Theorem 1.10]{OS12} that the condition $\alpha_G(X) > n/(n + 1)$ implies $K$-stability of $X$ along $G$-equivariant degenerations (so-called $G$-equivariant $K$-stability, see Definition~ \ref{defi:Kstab}).

Automorphism groups have been successfully used to establish the existence (or nonexistence) of K\"ahler--Einstein metrics for many particular examples of Fano varieties. To list a few examples, we should mention the obstructions for the existence of such metrics \cite{Mat57, Fut83}. Also, for smooth del Pezzo surfaces the necessary and sufficient condition for being K\"ahler--Einstein is reductivity of $\aut(X)$ \cite{Ti90}. For toric Fano varieties $K$-stability was studied in \cite{WZ04, Don02, Don08}. The case of varieties with torus action of complexity one was considered in \cite{Su13, IS17}. In fact, by \cite[Theorem 1.4]{LWX18} torus-equivariant $K$-polystability of a Fano variety $X$ with a torus action is equivalent to $K$-polystability of $X$. Tian's result was applied in \cite{Nad90, C08, CS09} to prove the K\"ahler--Einstein property for certain Fano threefolds, including those of types $V_1, V_5$ (see e.g. \cite[\S 12]{IP99} for the classification of Fano threefolds). Equivariant $K$-polystability of Fano threefolds of type $V_{22}$ was studied in \cite[Section 5]{Don08},  \cite[Appendix ~A]{CS12} and in \cite{CS18}. For a criterion of equivariant $K$-stability of spherical Fano varieties see \cite{Del16}. Also an extremely important general result was proved by Datar and Sz\'ekelyhidi.

\begin{thm}[{\cite[Theorem 1]{DS16}}] \label{thm:DS} Let $X$ be a Fano manifold and $G \subset \aut(X)$ a reductive subgroup. If $X$ is $K$-polystable with respect to $G$-equivariant test configurations then $X$ is K\"ahler--Einstein.
\end{thm}

These results show the significance of equivariant $K$-stability for the study of Fano varieties.

Another invariant, the so-called $\delta$-invariant, was defined for Fano varieties by Fujita and Odaka \cite[Definition 0.2]{FO18} using log canonical thresholds of basis-type divisors. The inequality $\delta(X) > 1$ implies uniform $K$-stability of $X$ by \cite[Theorem 0.3]{FO18} and in fact turns out to be equivalent to it (see \cite[Theorem B]{BlJ20}). 

Since uniform $K$-stability forces the automorphism group of $X$ to be finite \cite[Theorem 5.4]{BBEGZ19}, we cannot use $\delta$-invariant to study K\"ahler--Einstein property of Fano varieties with $\aut(X)$ infinite. However, if we restrict to $G$-equivariant degenerations for a suitable subgroup $G \subset \aut(X)$ (which is sufficient by Theorem \ref{thm:DS}), then in some examples (see e. g. Example \ref{exa:pgl}) uniform $K$-stability holds. Therefore it is reasonable to expect that there is a version of $\delta$-invariant for a variety $X$ with a reductive subgroup $G \subset \aut(X)$. 

The main goal of the present paper is to define the $\delta_G$-invariant for a pair $(X, L)$ where $X$ is a projective variety, $L$ is a big $\mathbb{Q}$-divisor and $G \subset \aut(X, L)$ is a closed connected subgroup. To do so, we follow the approach to $\delta$-invariants and $K$-stability via valuations on the field of rational functions on $X$, developed in \cite{Fuj19a, Fuj19b, BlJ20, BHJ17}. Note also that an alternative valuative criterion for $K$-semistability was proved in \cite{Li18}. In Section 2 we give the necessary definitions. The main object we consider is the space $\dival^G_X$ of $G$-invariant divisorial valuations on $X$. Up to a multiplicative constant, every such valuation $v$ is given by the order of vanishing at the generic point of a $G$-stable prime divisor $E$ on a birational model $\varphi \colon Y \to X$. We consider the log discrepancy $A_X(v) = 1 + \ord_E(K_{Y/X})$ and the expected vanishing order $$S_L(v) = \frac{1}{\Vol(L)}\int_0^{\infty}\Vol(\varphi^*L - tE)dt$$ of the valuation. Being inspired by \cite[Theorem 4.4]{BlJ20}, we define (see Definition \ref{defi:MainDefi}): $$\delta_G(X, L) = \inf_{v \in \dival^G_X}\frac{A_X(v)}{S_L(v)}.$$ Also in Section 2 we study basic properties of this invariant and compare it to $\alpha_G$-invariant of Tian. In Section $3$ we discuss equivariant $K$-stability and prove our main result.

\begin{thm}[see Theorem \ref{thm:Main} below] \label{thm:Main1} Let $(X, -K_X)$ be a klt $\mathbb{Q}$-Fano variety with the anticanonical polarization. Let $G \subset \aut(X)$ be a closed connected subgroup. Then $(X, -K_X)$ is uniformly $K$-stable (resp. $K$-semistable) with respect to $G$-equivariant degenerations if and only if the $\delta_G$-invariant of $(X, -K_X)$ is greater than one (resp. greater or equal to one). 
\end{thm}

It is of course desirable to generalize this theorem to any closed subgroup $G \subset \aut(X)$ (e. g. a finite group $G$). In Section $3$ we also establish a further connection of $\delta_G(X)$ with metric invariants of $X$, see Proposition \ref{prop:Beta} below.

In Section $4$ we investigate $\delta$-invariants of varieties with an action of a torus $T = (\mathbb{G}_m)^k$. We prove that $\delta$-invariant can be computed using only $T$-invariant valuations (Proposition \ref{prop:Solv}), generalizing a result of Blum and Jonsson, who considered the case of a toric variety $X$ and a maximal torus $T$.

Section $5$ is devoted to $\delta_G$-invariants of spherical Fano varieties; we refer to this section for basic definitions and notation from the theory of spherical varieties. If $X$ is a spherical Fano variety under the action of a connected reductive group $G$, we give a formula for $\delta_G(X)$. We choose a Borel subgroup $B \subset G$ and a maximal torus $T \subset B$. This formula uses the description of $\Val^G_X$ as the cone $\mathcal{V}$ in a finite-dimensional vector space. The log discrepancy $A_X(v)$ identifies with a certain piecewise linear function $h$ on the cones $C \subset \mathbb{F}_X$ in the complete colored fan $\mathbb{F}_X$ of $X$. The function $S_{-K_X}(v)$ can be expressed, following \cite{Del16}, via the moment polytope $\Delta^+$, the Duistermaat--Heckman measure $\mathrm{DH}$ on $\Delta^+$ (see \cite[Theorem 4.5]{Del16}) and the vector $2\rho_Q$, determined by $\Delta^+$ and the root system of $(G; T)$.

\begin{prop}[see Proposition \ref{prop:Form} below] \label{prop:Form1} Let $X$ be a Fano variety which is spherical under the action of a connected reductive group $G$. Then $\delta_G$-invariant of $X$ can be expressed as follows: $$\delta_G(X) = \min_{\substack{\mathrm{ord}_{D_i} \in \mathcal{V} \cap \mathcal{C} \\ \mathcal{C} \in \mathbb{F}_X}}\frac{a_{D_i}}{a_{D_i} - V \cdot \langle2\rho_Q - \mathrm{bar}_{DH}(\Delta^+), \pi^{-1}(\mathrm{ord}_{D_i}) \rangle}.$$ Here $\mathrm{bar}_{DH}(\Delta^+)$ is the barycenter of $\Delta^+$ with respect to the Duistermaat--Heckman measure $\mathrm{DH}$ and $V$ is a constant depending on $\Delta^+$ and $\mathrm{DH}$ only. The minimum is taken over a finite set $\mathrm{ord}_{D_1}, \ldots, \mathrm{ord}_{D_N}$ of valuations corresponding to generators of one-dimensional subcones (edges) in $\mathcal{C} \cap \mathcal{V}, \mathcal{C} \in \mathbb{F}_X$.
\end{prop}

In Section $6$ we consider the case of a variety with an action of a finite group $G$; aiming for possible generalization of \cite[Theorem 1.2]{Su13}, we give an alternative definition of $\delta_G$ using $G$-invariant divisors and prove the ramification formula. 

\begin{prop}[see Proposition \ref{prop:Ramif} below] \label{prop:Ramif1} Let $X$ be a variety with klt singularities and $-K_X$ big. Let $G \subset \aut(X)$ be a finite group. Denote by $Y = X / G$ the quotient variety and let $B = \sum_i(1 - 1/m_i)B_i$ be the branch divisor on $Y$. Then we have $$\delta_G(X) = \delta(Y, B)$$ where $\delta(Y, B)$ is the $\delta$-invariant of the klt pair $(Y, B)$.
\end{prop}

We expect that there is a unified definition of $\delta_G$ for any closed subgroup $G \subset \aut(X, L)$. We also hope for a generalization of Theorem \ref{thm:Main} to the case of a Fano variety $X$ with an arbitrary closed subgroup $G \subset \aut(X)$.

\textbf{Acknowledgement.} We would like to thank our advisor Constantin Shramov for his patience and support. We thank Harold Blum for careful reading of a preliminary version of this text and pointing out numerous inaccuracies. We thank Kento Fujita, Thibaut Delcroix and Yuji Odaka for many helpful remarks; we also thank the referee for careful reading of the paper. The author is partially supported by Laboratory of Mirror Symmetry NRU HSE, RF Government grant, ag. 14.641.31.0001.

\bigskip

\textbf{Note added to the updated version} Soon after the first version of this paper had appeared on arXiv, we have been informed that Ziwen Zhu \cite{Zhu19} gave an independent proof of a $G$-equivariant valuative criterion for $G \subset \aut(X)$ arbitrary (not nesessarily connected). He replaced $\dival^G_X$ by the space of finite-orbit divisorial valuations (invariant under the connected component $G^0$) and the functions $\beta_X(v)$ and $S_{-K_X}(v)$ by their $G$-analogues. Using this criterion, Yuchen Liu and Ziwen Zhu defined in \cite{LZ20} the invariant $\delta_G$ and generalized our Theorem \ref{thm:Main} for any algebraic subgroup $G \subset \aut(X)$.

Moreover, at the same time Chi Li \cite{Li19} has given a criterion for $G$-uniform $K$-stability of a Fano variety $(X, -K_X)$ for $G \subset \aut(X)$ connected. The notion of $G$-uniform $K$-stability was introduced by Hisamoto \cite{His16}. It is weaker than the notion of $G$-equivariant uniform $K$-stability since it replaces the usual norm $||(\XX, \LL)||$ on test configurations by the norm $||(\XX, \LL)||_{\mathbb{T}}$ orthogonal to the identity component of the center $\mathbb{T} = C(G)^0 \subset \aut^0(X)$. He also proved that for $G = \aut^0(X)$ or $G = T \subset \aut(X)$ a maximal torus this condition is equivalent to existence of a K\"ahler--Einstein metric on $X$.

\bigskip

\section{General definitions} 

\subsection{Notation and conventions} We work over the field $\mathbb{C}$ of complex numbers. A $\mathbb{Q}$-Fano variety is a projective variety such that $K_X$ is an ample $\mathbb{Q}$-Cartier divisor. We restrict ourselves to $\mathbb{Q}$-Fano varieties (or pairs) with Kawamata log terminal (klt) singularities. For all basic information regarding singularities we refer to \cite{Kol97}. A $\mathbb{Q}$-line bundle $L$ is a reflexive sheaf of rank $1$ such that some tensor power of $L$ is locally free.

In this section we recall the definitions of  $\alpha$ and $\delta$-invariants via log canonical thresholds and valuations. Then we study the space of $G$-invariant valuations and define $\delta_G$.

\subsection{Log canonical thresholds}

\begin{defi} Let $X$ be a normal projective variety and let $\Delta$ be an effective $\QQ$-divisor on $X$ such that $K_X + \Delta$ is $\QQ$-Cartier and the pair $(X, \Delta)$ has klt singularities. For an effective $\QQ$-Cartier $\QQ$-divisor $D$ on $X$ we define the {\em log canonical threshold} of $D$ with respect to $(X, \Delta)$ by the formula $$\lct(X, \Delta; D) = \sup\{t \in \mathbb{R} \mid \mbox{$(X, \Delta + tD)$ is log canonical}\}.$$ 
\end{defi}

\begin{defi} \label{defi:Divisors} Let $(X, \Delta)$ be a klt log Fano pair. We define the {\em $\alpha$-invariant} of $(X, \Delta)$ by $$\alpha(X, \Delta) = \inf\{\lct(X, \Delta; D) \mid D \Qlin -(K_X + \Delta) \mbox{ and $D$ is an effective $\mathbb{Q}$-divisor} \}.$$ We also define the {\em $\delta$-invariant} of the pair $(X, \Delta)$ as follows. For every $m \in \mathbb{N}$ such that $m(K_X + \Delta)$ is a Cartier divisor we look at the space $H^0(X, -m(K_X + \Delta))$; put $N_m = h^0(X, -m(K_X + \Delta))$. For every basis $(s_1, \ldots, s_{N_m})$ of $H^0(X, -m(K_X + \Delta))$ we denote $D(s_1), \ldots, D(s_{N_m})$ the corresponding divisors. We call effective $\QQ$-divisors $\mathbb{Q}$-linearly equivalent to $-K_X$ which have the form $$D = \frac{1}{mN_m}(D(s_1) + \ldots + D(s_{N_m}))$$ (anticanonical) $\QQ$-divisors {\em of $m$-basis type} on $(X, \Delta)$. We define for $m \in \mathbb{N}$ the invariant $\delta_m(X, \Delta)$ by $$\delta_m(X, \Delta) = \inf\{\lct(X, \Delta; D) \mid \mbox{$D$ is of $m$-basis type}\}$$ and finally we define the $\delta$-invariant by $$\delta(X, \Delta) = \limsup_{m \in \mathbb{N}}\delta_m(X, \Delta).$$ If $\Delta = 0$ then we simply write $\alpha(X)$ for $\alpha(X, 0)$ and analogously for $\delta(X)$.
\end{defi}

\subsection{The space of valuations} In this subsection we recollect some basic information about the space of valuations on the function field of a variety. We refer to \cite{JM12, BdFFU15, BlJ20} for more details. In this subsection $X$ is a normal and $\QQ$-Gorenstein projective variety over $\mathbb{C}$.

\begin{defi}\label{defi:Val} A valuation on $X$ is a real valuation $v \colon \mathbb{C}(X)^* \to \mathbb{R}$ on the function field of $X$ which is trivial on $\mathbb{C}$. We denote by $\Val_X$ the set of all nontrivial valuations on $X$. The latter is endowed with the topology of pointwise convergence, i. e. the weakest topology with the property that all evaluation maps $\mathrm{ev}_f \colon \mathrm{Val}_X \to \mathbb{R}, v \mapsto v(f)$ are continuous for every $f \in \mathbb{C}(X)^*$.
\end{defi}

\begin{defi}\label{defi:Dival} A valuation on $X$ is called divisorial if it has the form $v = c\cdot \ord_E(\cdot)$ where $E$ is a prime divisor on a birational model of $X$ and $c \in \mathbb{R}$. The set of divisorial valuations on $X$ is denoted by $\dival_X$.
\end{defi}

\begin{prop}\label{prop:Dense} The set of divisorial valuations is dense in the space of all valuations on $\mathbb{C}(X)$ in the topology of pointwise convergence.
\end{prop}

\begin{rmk}\label{rmk:Berk} The above Proposition \ref{prop:Dense} is proved using the theory of Berkovich spaces. We sketch the proof for reader's convenience. We can associate to the variety $X$ its Berkovich analytification $X^{an}$. The set of points of $X^{an}$ is, by definition, the set of pairs $(x, |\cdot|_x)$ where $x \in X$ is a scheme point and $|\cdot|_x \colon k(x) \to \mathbb{R}$ is a valuation on the residue field of $x$. Then the set $\Val_X$ is identified with the preimage of the generic point of $X$ under the projection $\pi \colon X^{an} \to X$. The topology induced on $\Val_X$ from $X^{an}$ is precisely the topology of pointwise convergence. Thus Proposition \ref{prop:Dense} follows from the density theorem for divisorial points in a Berkovich space (see e.g. \cite[Theorem 7.12]{Gub98}).
\end{rmk}

\begin{defi}\label{defi:LogDis} Let $v = \ord_E$ be a divisorial valuation on $X$ where $E \subset Y$ is a prime divisor on a birational model $f \colon Y \to X$. Let $K_{Y/X} = K_Y - f^*K_X$ be the relative canonical divisor. Then the log discrepancy of $v$ is defined by $A_X(v) = 1 + \ord_E(K_{Y/X})$.
\end{defi}

\begin{prop}\label{prop:Extend}(\cite[Theorem 3.1]{BdFFU15}) The log discrepancy extends to a function $$A_X \colon \Val_X \to \mathbb{R} \cup \{\infty\}$$ which is lower semicontinuous and homogeneous of order $1$, that is, $A_X(\lambda \cdot v) = \lambda \cdot A_X(v)$ for $\lambda \in \mathbb{R}_{\geqslant 0}$.
\end{prop}

We recall the definition of the volume function for $\mathbb{Q}$-divisors (see e.g. \cite[Section 2.2.C]{Laz04}).

\begin{defi} Let $D$ be a $\mathbb{Q}$-divisor on a variety $X$ of dimension $n$. Let $k \in \mathbb{N}$ be such that $kD$ is an integral Cartier divisor. Then the volume of $D$ is defined by $$\Vol(D) = \limsup_{m \to \infty}\frac{n!}{k^n \cdot m^n}h^0(X, \OO_X(mkD)).$$
\end{defi}

\begin{defi}\label{defi:VanOrd} Let $v = \ord_E$ be a divisorial valuation on $X$ where $E \subset Y$ is a prime divisor on a birational model $f \colon Y \to X$. Let $L$ be a big $\mathbb{Q}$-divisor on $X$. We define the pseudoeffective threshold of $v$ (or the maximal vanishing order) with respect to $L$ by $$T_L(v) = \sup\{t \in \mathbb{R}  |  \Vol(f^*L - tE) > 0\}.$$ We also define the expected vanishing order of $v$ with respect to $L$ by $$S_L(v) = \frac{1}{\Vol(L)}\int_0^{\infty}\Vol(f^*L - tE)dt.$$
\end{defi}

\begin{prop}\label{prop:Extend2} \cite[Proposition 3.13]{BlJ20} The functions $T_L$ and $S_L$ can be uniquely extended to functions $\Val_X \to \mathbb{R}$ which are lower semicontinuous and homogeneous of order $1$, that is, $T(\lambda \cdot v) = \lambda \cdot T(v)$ and $S(\lambda \cdot v) = \lambda \cdot S(v)$ for $\lambda \in \mathbb{R}_{\geqslant 0}$.
\end{prop}

Now we give the valuative definition of the $\alpha$ and $\delta$-thresholds, following \cite{BlJ20}. This definition is equivalent to Definition \ref{defi:Divisors}, as shown in \cite[Theorem C]{BlJ20}.

\begin{defi}\label{defi:Thresholds} We define the $\alpha$- and $\delta$-thresholds of $(X, L)$ (or the $\alpha$- and $\delta$-invariants of $(X, L)$) by the formulas $$\alpha(X, L)  = \inf_{v \in \Val_X}\frac{A_X(v)}{T_L(v)} = \inf_{v \in \dival_X}\frac{A_X(v)}{T_L(v)}$$ and analogously $$\delta(X, L)  = \inf_{v \in \Val_X}\frac{A_X(v)}{S_L(v)} = \inf_{v \in \dival_X}\frac{A_X(v)}{S_L(v)}.$$ In the case of a variety with a boundary divisor $(X, \Delta; L)$ such that $(X, \Delta)$ is klt we define $\alpha(X, \Delta; L)$ and $\delta(X, \Delta; L)$ in the same way by introducing the log discrepancy $A_{X, \Delta}(v)$ with respect to the pair $(X, \Delta)$.
\end{defi}

\subsection{Automorphisms preserving a big divisor class} We consider a pair $(X, L)$ where $X$ is a normal complex projective variety and $L$ is a big $\QQ$-Cartier divisor. We denote by $\aut(X, L) \subset \aut(X)$ the subgroup of automorphisms of $X$ preserving the class $[L] \in \mathrm{Cl}(X)$. The following proposition is well-known; however we do not know a standard reference (see e. g. \cite[Lemma 3.1.2]{KPS18}).

\begin{prop}\label{prop:Linear} Let $(X, L)$ be a variety with a big $\mathbb{Q}$-Cartier divisor and let $m \in \mathbb{N}$ be such that $mL$ is Cartier and the map $$\varphi_{|mL|} \colon X \to \PP(H^0(X, \OO_X(mL))^*)$$ is birational onto its image. Then the above map is $\aut(X, L)$-equivariant and the group $\aut(X, L)$ embeds into $\mathrm{PGL}(H^0(X, \OO_X(mL)^*)$. Therefore $\aut(X, L)$ is a linear algebraic group.
\end{prop}

We are mostly interested in the case when $X$ is a Fano variety and $L = -K_X$. Since the automorphism group $\aut(X)$ preserves the anticanonical class, we have the following corollary.

\begin{cor}\label{cor:FanoLin} The automorphism group of a $\QQ$-Fano variety is a linear algebraic group.
\end{cor}

We will also consider triples $(X, \Delta; L)$; here the group $\aut(X, \Delta; L) \subset \aut(X)$ is a stabilizer of the boundary divisor $\Delta$ and the class of $L$. By Proposition \ref{prop:Linear} the group $\aut(X, \Delta; L)$ is a linear algebraic group as well.

\subsection{$G$-invariant valuations}

In this subsection we work with $(X, L)$ as above and a closed connected subgroup $G \subset \aut(X, L)$; then $G$ is a linear algebraic group. The group $G$ acts on the function field $\mathbb{C}(X)$ by $f \mapsto \gamma \cdot f$ and therefore on $\Val_X$. We describe the space of $G$-invariant valuations mainly following \cite{Kn93, Tim11}.

\begin{defi} We denote by $\Val^G_X \subset \Val_X$ the set of all valuations $v \colon \mathbb{C} (X)^* \to \mathbb{R}$ which are invariant under the action of $G$. That is, a nontrivial valuation $v$ belongs to $\Val^G_X$ if and only if for all $\gamma \in G$ and for all $f \in \mathbb{C}(X)^*$ we have $v(\gamma \cdot f) = v(f).$ We also denote by $\dival^G_X$ the set of all {\it divisorial} $G$-invariant valuations. It has a topology induced from $\Val_X$.
\end{defi}

\begin{prop}{(see e. g. \cite[Proposition 19.8]{Tim11})} Let $X$ be a $G$-variety with $G$ connected. Then every $G$-invariant divisorial valuation $v \in \dival^G_X$ is proportional to a valuation $\ord_D$ where $D$ is a $G$-invariant prime divisor on some birational $G$-model of $X$.
\end{prop}

\begin{rmk}\label{rmk:Empty} Observe that the trivial valuation $v_{0} \colon \mathbb{C}(X) \to \mathbb{R}^*$ does not belong to $\Val_X$ or to $\Val^G_X$. Thus for any $v \in \Val^G_X$ the center of $v$ on $X$ (which exists since $X$ is projective) is $G$-invariant; its closure is a {\em proper} $G$-invariant subvariety of $X$. Hence the space $\Val^G_X$ can be empty, for example if $G$-action on $X$ is transitive (see Remark \ref{rmk:Homog} below). 
\end{rmk}

The following approximation theorem of Sumihiro is crucial for the study of $G$-invariant valuations.

\begin{thm}[{\cite[Lemma 10]{Sum74}}] \label{thm:Sumihiro} Let $G$ be a connected algebraic group and let $X$ be an irreducible $G$-variety. Then for any divisorial valuation $v$ on $\mathbb{C}(X)$ there exists a $G$-invariant valuation $\bar v \in \dival^G_X \cup \mathbb{R}_{\geqslant 0}\{v_0\}$ with the following property: for any $f \in \mathbb{C}(X)$ there exists a nonempty Zariski-open subset $U_f \subset G$ such that for any $\gamma \in U_f$ we have $v(\gamma \cdot f) = \bar v(f)$.
\end{thm}

\begin{prop}\label{prop:DivDense} Let $X$ be a $G$-variety. The set $\dival^G_X$ of $G$-invariant divisorial valuations is dense in the set $\Val^G_X$ in the topology of pointwise convergence. 
\end{prop}
\begin{proof} We assume that $\Val^G_X$ is nonempty; otherwise there is nothing to prove. Suppose that $v \in \Val_X$ is $G$-invariant. By Proposition \ref{prop:Dense} we can find a sequence $\{v_i\}$ of divisorial valuations converging to $v$. For every $i \in \mathbb{N}$ let $\bar v_i$ be the $G$-invariant valuation associated to $v_i$ by Theorem \ref{thm:Sumihiro}. It suffices to prove that the sequence $\{\bar v_i\}$ converges to $v$; since $v$ is nontrivial, almost all $v_i$ have to be divisorial in this case. Take a rational function $f \in \mathbb{C}(X)$; then for every $i \in \mathbb{N}$ we have the subset $$U_{f, i} = \{\gamma \in G \mid v_i(\gamma \cdot f) = \bar v_i(f)\}.$$ Let $U_f = \cap_{i \in \mathbb{N}}U_{f, i}$ be the intersection of these subsets; it is nonempty since the field $\mathbb{C}$ is uncountable. Then for any $\gamma \in U_f$ we obtain that $$\bar v_i(f) = v_i(\gamma \cdot f)$$ converge to $v(\gamma \cdot f) = v(f)$ by definition of $\bar v_i$ and $G$-invariance of $v$. Therefore for any $f \in \mathbb{C}(X)$ the sequence $\bar v_i(f)$ converges to $v(f)$, as desired. This proves density of $G$-invariant divisorial valuations in the space $\Val^G_X$.
\end{proof}

Now let us introduce the definitions for the $\alpha_G$ and $\delta_G$-thresholds via $G$-invariant valuations.

\begin{defi}\label{defi:MainDefi} Let $X$ be a variety and let $L$ be a big $\mathbb{Q}$-Cartier divisor. Let also $G \subset \aut(X, L)$ be a closed connected subgroup; we define the $G$-invariant $\alpha$- and $\delta$-thresholds of $(X, L)$ (we will also call them the $\alpha$- and $\delta$-invariants of $(X, L)$ with the action of $G$) by the formulas $$\alpha_G(X, L)  = \inf_{v \in \Val^G_X}\frac{A_X(v)}{T_L(v)} = \inf_{v \in \dival^G_X}\frac{A_X(v)}{T_L(v)}$$ and analogously $$\delta_G(X, L)  = \inf_{v \in \Val^G_X}\frac{A_X(v)}{S_L(v)} = \inf_{v \in \dival^G_X}\frac{A_X(v)}{S_L(v)}.$$ 
For a variety $X$ with $L = -K_X$ big we write $\delta_G(X, -K_X) = \delta_G(X)$ and the same for $\alpha_G$. In the case of a pair $(X, \Delta; L)$ and a subgroup $G \subset \aut(X, \Delta; L)$ we can define $\alpha_G(X, \Delta; L)$ and $\delta_G(X, \Delta; L)$ in the same way by introducing the log discrepancy $A_{X, \Delta}(v)$ with respect to the pair $(X, \Delta)$.
\end{defi}

\begin{rmk} The extension procedure in Propositions \ref{prop:Extend} and \ref{prop:Extend2} uses density of divisorial valuations in $\Val_X$. By Proposition \ref{prop:DivDense} we can restrict the functions $A_X(v)$, $T_L(v)$ and $S_L(v)$ to $\dival^G_X$ and extend to $\Val^G_X$ by the same procedure; the resulting functions will coincide with the restrictions to $\Val^G_X$ of the respective extensions of $A_X(v)$, $T_L(v)$ and $S_L(v)$, given by Propositions \ref{prop:Extend} and \ref{prop:Extend2}.
\end{rmk}

\begin{rmk} \label{rmk:Homog} If a variety $X$ is a homogeneous space under the action of an algebraic group $G$ then both $\alpha_G$ and $\delta_G$ are infinite. Indeed, since the action of $G$ is transitive, the space $\Val^G_X$ is empty.
\end{rmk}

\begin{rmk}\label{rmk:DefAlpha} In \cite[Appendix A]{CS08} the algebraic version of Tian's invariant $\alpha_G(X)$ was defined by $$\alpha_G(X, L) = \inf_{m \in \mathbb{N}} \inf_{\substack{ |\Sigma| \subset |mL| \\  |\Sigma|^G = |\Sigma|}} \lct(X, \frac1m|\Sigma|)$$ where the second infimum is taken over all $G$-invariant linear subsystems $|\Sigma| \subset |mL|$.
\end{rmk}

\begin{prop}\label{prop:Alpha} The definition of $\alpha_G(X, L)$ given in Remark \ref{rmk:DefAlpha} is equivalent to Definition \ref{defi:MainDefi}.
\end{prop}

\begin{proof} By definition from Remark \ref{rmk:DefAlpha}, we have $\alpha_G(X, L) = \inf_{m \in \mathbb{N}}\alpha_{G, m}(X, L)$ where $$\alpha_{G, m}(X, L) = \inf_{\substack{|\Sigma| \subset |mL| \\ |\Sigma|^G = |\Sigma|}} \lct(X, \frac1m|\Sigma|).$$ Note that $\lct(X, |\Sigma|)$ is defined as the log canonical threshold of the base ideal $\mathfrak{b}_{\Sigma}$ of $|\Sigma|$. Since $|\Sigma|$ is $G$-invariant, the base ideal is also $G$-invariant. There exists a divisorial valuation computing $\lct(X, \mathfrak{b}_{\Sigma})$ by \cite[Section 1.7]{BlJ20}; moreover, by $G$-equivariant log resolution (see e.g. \cite[Section 3.9.1]{Kol07}) and connectedness of $G$ this valuation can be chosen to be $G$-invariant. Note also that for $v \in \Val^G_X$ we have $$\sup_{m \in \mathbb{N}} \sup_{\substack{|\Sigma| \subset |mL| \\ |\Sigma|^G = |\Sigma|}}\frac1m v(\mathfrak{b}_{\Sigma}) 
= T_{L}(v).$$ Indeed, by definition, $T(v) = \sup_{m \in \mathbb{N}} \sup_{D \in |mL|}\frac1m v(D)$. For any $D_0 \in |mL|$ with $v(D_0) = \lambda$ the linear system $$|\Sigma| = \{D \in |mL| \mid v(D) \geqslant \lambda/m\}$$ is nonempty and $G$-invariant by invariance of $v$; moreover, $$v(\frac1m |\Sigma|) = \inf_{D \in \mid \Sigma|}v(D) \geqslant \lambda/m.$$ The reverse inequality is obvious. Expanding the definitions and switching the order of infima we can write $$\inf_{m \in \mathbb{N}} \inf_{\substack{|\Sigma| \subset |mL| \\ |\Sigma|^G = |\Sigma|}} \lct(X, \frac1m|\Sigma|) = \inf_{m \in \mathbb{N}} \inf_{\substack{|\Sigma| \subset |mL| \\ |\Sigma|^G = |\Sigma|}}\inf_{v \in \Val^G_X}\frac{A_X(v)}{\frac1m v(\mathfrak{b}_{\Sigma})} = \inf_{v \in \Val^G_X}\frac{A_X(v)}{T_{L}(v)}$$ and thus we obtain the equivalence of two definitions.
\end{proof}

\begin{prop}\label{prop:Properties} Let $(X, L)$ be a variety with a big $\mathbb{Q}$-divisor and let $G \subset \aut(X, L)$ be a closed connected subgroup. We have the following inequalities for $\alpha_G(X, L)$ and $\delta_G(X, L)$ where $X$ has dimension ~$n$: $$0 < \alpha_G(X, L) \leqslant \delta_G(X, L) \leqslant (n+1)\alpha_G(X, L).$$ If, moreover, $L$ is ample then we have a stronger inequality $$\alpha_G(X, L)(1 + \frac1n) \leqslant \delta_G(X).$$
\end{prop}

\begin{proof} The first inequalities follow from \cite[Lemma 2.6]{BlJ20}. The second inequality follows from the fact that $\frac{S_L(v)}{T_L(v)} \leqslant \frac{n}{n+1}$ for $L$ an ample $\mathbb{Q}$-divisor and all valuations $v \in \Val_X$ \cite[Proposition 2.1]{Fuj19b}.
\end{proof}

\section{Equivariant $K$-stability} In this section we collect basic information on $G$-equivariant $K$-stability for a variety $X$ with an ample polarization $L$. Then we prove the main Theorem \ref{thm:Main1} (see Theorem \ref{thm:Main} below). 

We call a pair $(X, L)$ with $X$ projective and $L$ an ample $\QQ$-divisor a polarized variety.

\subsection{Basic definitions related to $K$-stability} The following definitions are well-known and can be found e.g. in \cite{Ti97, Don02, BHJ17}.

\begin{defi} A test configuration for a polarized variety $(X, L)$ with $L$ ample is a pair $(\XX, \LL)$ consisting of a variety $\XX$ with a projective surjective morphism $\pi \colon \XX \to \mathbb{A}^1$ and a $\pi$-semiample $\QQ$-line bundle $\LL$ together with a $\mathbb{G}_m$-action on $\XX$ preserving the class of $\LL$ such that \begin{itemize} \item The morphism $\pi \colon \XX \to \mathbb{A}^1$ is $\mathbb{G}_m$-equivariant with respect to the given action on $\XX$ and the multiplicative action on $\mathbb{A}^1$; \item The pair $(\XX \setminus \XX_0, \LL|_{\XX \setminus \XX_0})$ is $\mathbb{G}_m$-equivariantly isomorphic to $(X \times (\mathbb{A}^1 \setminus \{0\}), p_1^*L)$.\end{itemize}
\end{defi}

\begin{defi}[see e. g. \cite{OS12, DS16}] Let $G$ be a closed subgroup of $\aut(X, L)$. A test configuration $(\XX, \LL)$ is called $G$-equivariant if there exists an action of $G$ on the pair $(\XX, \LL)$ which commutes with the $\mathbb{G}_m$-action and restricts to the given $G$-action on $(\XX_t, \LL|_{\XX_t}) \simeq (X, L)$ for $t \neq 0$.
\end{defi}

\begin{defi} Let $(\XX, \LL)$ be a test configuration for $(X, L)$. The compactification $(\bar \XX, \bar \LL)$ of $(\XX, \LL)$ is the degeneration of $(X, L)$ over $\mathbb{P}^1$ defined by gluing $(\XX, \LL)$ with $(X \times (\mathbb{P}^1 \setminus \{0\}), p_1^*L)$ along the subset $X \times (\mathbb{P}^1 \setminus \{0, 1\})$. The compactification is $G$-equivariant if $(\XX, \LL)$ is.
\end{defi}

\begin{defi}[see e. g. \cite{Der16, LX14}] A test configuration $(\XX, \LL)$ is trivial if the relative canonical model $(\XX^{\can}, \LL^{\can})$ of the normalization of $(\XX, \LL)$ is $\mathbb{G}_m$-equivariantly isomorphic to $(X \times \mathbb{A}^1, p_1^*L + c\XX_0)$ for some $c \in \mathbb{Q}$. We say that $(\XX, \LL)$ is of product type if $\XX^{\can}$ is isomorphic to $X \times \mathbb{A}^1$.
\end{defi}

\begin{defi}[\cite{BHJ17, Der16, LX14}] Consider a test configuration $(\XX, \LL)$ for $(X, L)$. Let us denote by $\mathcal{Z}$ the normalization of the graph of the map $\bar \XX \dasharrow X \times \mathbb{P}^1$; it has natural maps $\pi \colon \mathcal{Z} \to X \times \mathbb{P}^1$ and $\varphi \colon \mathcal{Z} \to \XX$. We define the following invariants: $$\lambda_{\max}(\XX, \LL) = \frac{(p_1 \circ \pi)^*L^n \cdot \varphi^*\bar \LL}{L^n} \qquad \mbox{and} \qquad  J^{NA}(\XX, \LL) = \lambda_{\max}(\XX, \LL) - \frac{\bar \LL^{n+1}}{(n+1)L^n}$$ where the latter is called the norm of $(\XX, \LL)$. Also for a normal test configuration $(\XX, \LL)$ for a $\mathbb{Q}$-Fano variety $(X, -K_X)$ we define the Donaldson--Futaki invariant $\mathrm{DF}(\XX, \LL)$ by $$\mathrm{DF}(\XX, \LL) = \frac{n}{n + 1}\cdot \frac{\bar \LL^{n+1}}{(-K_X)^n} + \frac{\bar \LL^n \cdot K_{\XX / \mathbb{P}^1}}{(-K_X)^n}.$$ We also define the Ding invariant of $(\XX, \LL)$ as follows. Let $D_{(\XX, \LL)}$ be the unique $\mathbb{Q}$-divisor defined by the conditions $\mathrm{Supp}(D_{(\XX, \LL)}) \subset \XX_0$ and $D_{(\XX, \LL)} \sim_{\mathbb{Q}} -K_{\XX/\mathbb{P}^1} - \bar \LL$. The Ding invariant $\mathrm{Ding}(\XX, \LL)$ is defined by $$\mathrm{Ding}(\XX, \LL) = -\frac{\bar \LL^{n+1}}{(n+1)(-K_X)^n} - 1 + \lct(\XX, D_{(\XX, \LL)}; \XX_0).$$
\end{defi}

\begin{rmk} The crucial property of the norm $J^{NA}$ is the following: this function is nonnegative, and it vanishes on $(\XX, \LL)$ if and only if the normalization of $(\XX, \LL)$ is a trivial test configuration (for a proof see e. g. \cite[Theorem 1.3]{Der16}). 
\end{rmk}

\begin{rmk} Examples of product-type test configurations which are not trivial are given by one-parameter subgroups $\mathbb{G}_m \subset \aut(X)$ acting diagonally on $\XX = X \times \mathbb{A}^1$.
\end{rmk}

\begin{defi}\label{defi:Kstab} A $\mathbb{Q}$-Fano variety $(X, -K_X)$ is called \begin{itemize} \item $K$-semistable if $\mathrm{DF}(\XX, \LL) \geqslant 0$ for every normal test configuration $(\XX, \LL)$ for $(X, -K_X)$; \item $K$-polystable if $\mathrm{DF}(\XX, \LL) \geqslant 0$ for every normal test configuration $(\XX, \LL)$ and $\mathrm{DF}(\XX, \LL) = 0$ only if $(\XX, \LL)$ is of product type; \item $K$-stable if $\mathrm{DF}(\XX, \LL) \geqslant 0$ for every normal test configuration $(\XX, \LL)$ and $\mathrm{DF}(\XX, \LL) = 0$ only if $(\XX, \LL)$ is trivial; \item uniformly $K$-stable if there exists $\varepsilon > 0$ such that $\mathrm{DF}(\XX, \LL) \geqslant \varepsilon J^{NA}(\XX, \LL)$ for every normal test configuration $(\XX, \LL)$ for $(X, -K_X)$.\end{itemize} Analogous notions for Ding-stability are defined in the same way with $\mathrm{DF}(\XX, \LL)$ replaced by $\mathrm{Ding}(\XX, \LL)$. For a closed subgroup $G \subset \aut(X)$ we say that $(X, -K_X)$ is $G$-equivariantly $K$-stable (or $K$-stable along $G$-equivariant test configurations) if the corresponding inequalities hold for $G$-equivariant test configurations; the same for uniform $K$- or Ding-stability.
\end{defi}

\begin{rmk} \label{rmk:Stab} We have the obvious implications: $$ \mbox{uniform $K$-stability $\Rightarrow$ $K$-stability $\Rightarrow$ $K$-polystability $\Rightarrow$ $K$-semistability}$$ and the same ones for $G$-equivariant stability. Note also that $K$-stability implies that $\aut(X)$ is discrete, whereas $K$-polystability allows $\aut(X)$ to be infinite. As a consequence, uniform $K$-stability is not equivalent to existence of K\"ahler--Einstein metrics. In fact, $G$-equivariant uniform $K$-stability is also stronger than being K\"ahler--Einstein in general.
\end{rmk}

\subsection{The main result} Let us now state and prove the main result of this work (Theorem \ref{thm:Main1} from the introduction).

\begin{thm}\label{thm:Main} Let $(X, L) = (X, -K_X)$ be a $\QQ$-Fano variety with the anticanonical polarization and let $G \subset \aut(X)$ be a connected subgroup. Then the following conditions are equivalent: \begin{enumerate} 
\item The variety $X$ is uniformly $K$-stable (resp. $K$-semistable) along $G$-equivariant test configurations;
\item The variety $X$ is uniformly Ding-stable (resp. Ding-semistable) along $G$-equivariant test configurations;
\item We have the inequality $\delta_G(X) > 1$ (resp. $\delta_G(X) \geqslant 1$).
\end{enumerate}
\end{thm}

 The proof of Theorem \ref{thm:Main} mainly follows \cite[Theorem 1.4]{Fuj19a} and \cite{LX14}. We divide the proof into Propositions \ref{prop:Fujita1} and \ref{prop:Fujita2} below. Our goal is to check that the constructions of test configurations used in \cite[Theorem 4.1]{Fuj19a} and \cite[Theorem 4]{LX14} can be made $G$-equivariantly, for $G \subset \aut(X)$ a connected subgroup. The next proposition uses the ``blowing up'' techniques to check $K$-stability (see \cite{Od13, OS12, Der16}) in $G$-equivariant setting.

\begin{prop}\label{prop:Fujita1} Let $(X, L) = (X, -K_X)$ be a $\QQ$-Fano variety and $G \subset \aut(X)$ a closed connected subgroup; assume that for all $G$-equivariant normal test configurations $(\XX, \LL)$ for $(X, -K_X)$ we have the inequality $$\mathrm{Ding}(\XX, \LL) \geqslant \varepsilon J^{NA}(\XX, \LL)$$ for some $\varepsilon > 0$. Then we have for every $G$-invariant divisorial valuation on $\mathbb{C}(X)$ the inequality $$\frac{A_X(v)}{S_{-K_X}(v)} \geqslant \frac{1}{1 - \varepsilon}.$$ In particular, we have $\delta_G(X) > 1$.
\end{prop}

\begin{proof} Let us first outline the idea of the proof of \cite[Theorem 4.1]{Fuj19a}. Starting from a divisorial valuation $v = \ord_E$, we construct a sequence of test configurations $(\XX^r, \LL^r)$ using {\em flag ideals} (introduced in \cite{Od13}) associated to the valuation $v$. Then from the computations in \cite[Claim 4.4]{Fuj19a} and \cite[Claim 2.5]{Fuj19b} we conclude that uniform boundedness from below of the Ding invariants $\mathrm{Ding}(\XX^r, \LL^r)$ implies the required bound on $A_X(v)/S(v)$.  We show that this construction can be made $G$-equivariantly provided that we start from a $G$-invariant divisorial valuation. We assume that $\Val^G_X$ is nonempty; otherwise there is nothing to prove.

Let $r_0 \in \mathbb{N}$ be the Cartier index of $K_X$. To a divisorial valuation $v = \ord_E$ where  $E$ is a $G$-stable prime divisor on a birational model $\varphi \colon Y \to X$ we associate the {\em order filtration} $\mathcal{F}_v$ on the graded algebra $$R(X, -r_0K_X) = \bigoplus_{m \in \mathbb{N}}H^0(X, \OO_X(-mr_0K_X)) = \bigoplus_{m \in \mathbb{N}}V_m.$$ The filtration is defined by $$\mathcal{F}_v^t V_m = H^0(Y,\OO_X( -mr_0\varphi^*K_X - tE)) \subset H^0(X, \OO_X(-mr_0K_X)).$$ It is a filtration by $G$-invariant linear subspaces of $R(X, -r_0K_X)$. To this filtration and to any given $t \in \mathbb{R}, m \in \mathbb{N}$ we can associate a nontrivial $G$-invariant ideal $$\II_{(m, t)} = \mathrm{Im}(\mathcal{F}^t_v V_m \otimes \OO_{X}(-mr_0K_X) \to \OO_X).$$ By \cite[Claim 4.2]{Fuj19a}, we have $$\mathcal{F}^t_v V_m = H^0(X, -mr_0K_X \cdot \II_{(m, t)}).$$ We now define for appropriate $e_{+}, e_{-} \in \mathbb{Z}$ and $r \in \mathbb{N}$ large enough (as in \cite[Theorem 4.1]{Fuj19a}) the flag ideal $\II_r \subset \OO_{X \times \mathbb{A}^1}$ by $$\II_r = \II_{(r, re_{+})} + \II_{(r, re_{+} - 1)} \cdot t + \cdots + \II_{(r, re_{-} + 1)}\cdot t^{r(e_{+} - e_{-}) - 1} + (t^{r(e_{+} - e_{-})}).$$ This ideal is invariant under the action of $G$ on $X\times \mathbb{A}^1$. We construct a test configuration by blowing up the flag ideal $\Phi_r \colon \mathrm{Bl}_{\II_r}(X \times \mathbb{A}^1) \to X \times \mathbb{A}^1$. Let $E_r$ be the exceptional divisor of the blow-up; we set $$\XX^r = \mathrm{Bl}_{\II_r}(X \times \mathbb{A}^1) \qquad \mbox{and} \qquad \LL^r = \Phi_r^*(-K_{X \times \mathbb{A}^1}) - \frac{1}{rr_0}E_r.$$ Then the $G$-action lifts to the blow-up leaving $\LL^r$ invariant; therefore $(\XX^r, \LL^r)$ is a $G$-equivariant test configuration for $(X, L)$. By assumption, for the (still $G$-equivariant) normalization $(\XX^{r, \nu}, \nu^*\LL^{r})$ of any test configuration $(\XX^r, \LL^r), r \in \mathbb{N}$ we have the inequality $$\mathrm{Ding}(\XX^{r, \nu}, \nu^*\LL^{r}) \geqslant \varepsilon J^{NA}(\XX^{r, \nu}, \nu^*\LL^{r})$$ for some fixed $\varepsilon > 0$. By the proof of \cite[Theorem 4.1]{Fuj19a} and \cite[Claims 2.4, 2.5]{Fuj19b} we can express $\mathrm{Ding}(\XX^{r, \nu}, \nu^*\LL^{r})$ and $J^{NA}(\XX^{r, \nu}, \nu^*\LL^{r})$ via the invariants of the valuation $v = \ord_E$ and obtain an inequality $$A_X(v)/S_{-K_X}(v) \geqslant 1/(1 - \varepsilon).$$ Since $v = \ord_E \in \dival^G_X$ was arbitrary, we obtain $$\delta_G(X, L) = \inf_{v \in \dival^G_X}\frac{A_X(v)}{S_L(v)} \geqslant \frac{1}{1 - \varepsilon} > 1$$ as it was to be shown.
\end{proof}

We recall here some important results of Li--Xu and Fujita. These results show that uniform $K$- or Ding-stability can be checked using only special test configurations. Moreover, the Donaldson-Futaki invariants of these test configurations can be expressed via invariants of valuations. 

\begin{prop}[{\cite[Theorem 4]{LX14} and \cite[Theorem 5.2]{Fuj19a}}] \label{prop:SpecialTC} Let $(\XX, \LL)$ be a normal test configuration for a $\mathbb{Q}$-Fano variety $(X, -K_X)$. Then there exist \begin{itemize} \item a finite base change $(\XX^{(d)}, \LL^{(d)}) \to (\XX, \LL)$; \item a $\mathbb{G}_m$-equivariant birational map $\XX^{(d)} \to \XX^s$ obtained by running a relative MMP with scaling of an ample divisor $\mathcal{H}$\end{itemize} such that the resulting test configuration $(\XX^s, \LL^s)$ is special, that is, $\XX^s_0$ is irreducible and reduced. Moreover, for every special test configuration $(\XX^s, \LL^s)$ obtained from $(\XX, \LL)$ the following properties hold true: \begin{itemize} \item [(a)]There exists $d \in \mathbb{N}$ such that $\mathrm{DF}(\XX^s, \LL^s) \leqslant d \cdot \mathrm{DF}(\XX, \LL)$ and such that for any $\varepsilon \in [0,1]$ we have $$\mathrm{Ding}(\XX^s, \LL^s) - \varepsilon J^{NA}(\XX^s, \LL^s) \leqslant d \cdot ( \mathrm{Ding}(\XX, \LL) - \varepsilon J^{NA}(\XX, \LL));$$  \item[(b)] The Donaldson--Futaki invariant $\mathrm{DF}(\XX^s, \LL^s)$ is equal to the Ding invariant of $(\XX^s, \LL^s)$ and can be expressed via the invariants of the divisorial valuation $v = v_{\XX^s_0}$ as follows: $$\mathrm{DF}(\XX^s, \LL^s) = A_X(v_{\XX^s_0}) - S_L(v_{\XX^s_0}).$$ \end{itemize}   
\end{prop}

The next proposition says that one can pass from any $G$-equivariant test configuration to a $G$-equivariant {\em special} test configuration.

Note that for any test configuration $(\XX, \LL)$ the $\mathbb{G}_m$-equivariant birational map $\XX \dasharrow X \times \mathbb{A}^1$ gives an isomorphism $\mathbb{C}(\XX) = \mathbb{C}(X)(t)$. This isomorphism is $G$-equivariant if the test configuration $(\XX, \LL)$ is. The projection $X \times \mathbb{A}^1 \to X$ gives a $G$-equivariant embedding $\mathbb{C}(X) \subset \mathbb{C}(X)(t)$. If the central fiber $\XX_0$ is irreducible then the restriction of the divisorial valuation $v_{\XX_0}$ to $\mathbb{C}(X)$ is either divisorial or trivial (see e.g. \cite[Lemma 4.1]{BHJ17} or \cite[Proposition B.8]{Tim11}). 

\begin{prop}\label{prop:Fujita2} Let $(\XX, \LL)$ be a normal test configuration for a $\mathbb{Q}$-Fano variety $(X, -K_X)$. If $(\XX, \LL)$ is $G$-equivariant for a closed connected subgroup $G \subset \aut(X)$ then every special test configuration $(\XX^s, \LL^s)$ constructed from $(\XX, \LL)$ as in Proposition \ref{prop:SpecialTC} is $G$-equivariant and satisfies the properties (a) and (b) from Proposition \ref{prop:SpecialTC}. In particular, if $(\XX^s, \LL^s)$ is non-trivial then the restriction of the divisorial valuation $v_{\XX^s_0}$ to the subfield $$\mathbb{C}(X) \subset \mathbb{C}(X)(t) \simeq \mathbb{C}(\XX^s)$$ is a $G$-invariant divisorial valuation.
\end{prop}

\begin{proof} We start from a normal $G$-equivariant test configuration $(\XX, \LL)$ and compactify it to $(\bar \XX, \bar \LL)$. Taking an equivariant log resolution of the pair $(\bar \XX, \bar \XX_0)$ (by \cite[Section 3.9.1]{Kol07}) and a $G$-equivariant finite base change $t \mapsto t^d$ (by \cite[Lemma 5]{LX14}) we can assume that the pair $(\bar \XX, \bar \XX_0)$ is log canonical. 

Then by Proposition \ref{prop:SpecialTC} a special test configuration can be obtained from $(\bar \XX, \bar \XX_0)$ by running a relative $K_{\bar \XX/\PP^1}$-MMP with scaling of an ample $\mathbb{Q}$-divisor $\mathcal{H}$. We can take $\mathcal{H}$ such that the class of $\mathcal{H}$ lies in ~ $\Pic^G(\bar \XX)$. We recall that $\overline{\mathrm{NE}}(\bar \XX) = \overline{\mathrm{NE}}(\bar \XX)^{G}$ by \cite[Lemma 1.5]{And01} since $G$ is connected. Therefore for every divisorial extremal ray of $\overline{\mathrm{NE}(X)}$ of the Mori cone the contraction is $G$-equivariant. For flips, the statement follows since the $\mathrm{Proj}$ of an algebra with a $G$-action has an induced structure of a $G$-variety. 

Thus at the final step of the MMP we obtain a $G$-equivariant test configuration $(\bar \XX^s, \bar \LL^s)$ such that $\XX^s_0$ is a $G$-stable prime divisor. So Proposition \ref{prop:SpecialTC} applies to $(\XX^s, \LL^s)$ and ensures that the properties (a) and (b) are satisfied. Moreover, the valuation $v_{\XX^s_0}$ on $\mathbb{C}(\XX^s)$ is $G$-invariant. If $(\XX^s, \LL^s)$ is nontrivial then the restriction of $v_{\XX^s_0}$ to $\mathbb{C}(X)$ is divisorial by \cite[Proposition B.8]{Tim11} or \cite[Lemma 4.1]{BHJ17}and invariant by the induced $G$-action on $X$.
\end{proof}
 
\begin{proof}[Proof of Theorem \ref{thm:Main}] The implication $(2) \Rightarrow (3)$ is precisely Proposition \ref{prop:Fujita1}. To show $(3) \Rightarrow (1)$ we use Propositions \ref{prop:SpecialTC} and \ref{prop:Fujita2} in order to pass from a $G$-equivariant test configuration $(\XX, \LL)$ to a special test configuration $(\XX^s, \LL^s)$ such that the valuation $v_{\XX^s_0}$ is $G$-invariant. We have the inequality $$A_X(v_{\XX_0^s})/S_{-K_X}(v_{\XX^s_0}) \geqslant 1/(1 - \varepsilon) > 1$$ for some $\varepsilon \in (0; 1)$ by assumption and thus $$\mathrm{DF}(\XX^s, \LL^s) = \mathrm{Ding}(\XX^s, \LL^s) \geqslant \varepsilon J^{NA}(\XX^s, \LL^s).$$ Therefore $\mathrm{DF}(\XX, \LL) \geqslant \varepsilon J^{NA}(\XX, \LL)$ by Proposition \ref{prop:SpecialTC}. Finally, since the Donaldson--Futaki and Ding invariants of special test configurations coincide, we get the implication $(1) \Rightarrow (2)$ The same implications for semistability follow by the same argument with $\varepsilon = 0$ in Propositions \ref{prop:Fujita1} and \ref{prop:SpecialTC}.
\end{proof}

\begin{rmk}\label{rmk:Rescale} We stated and proved the above theorem for a $\mathbb{Q}$-Fano variety $X$ with the anticanonical polarization. However, it can easily be seen that Theorem \ref{thm:Main} is still true if we rescale the polarization by a rational number $t$. Indeed, the Donaldson--Futaki invariant does not change under rescaling the polarization (see \cite[Definition 3.6]{BHJ17}). The norm function is rescaled by a multiple of $t$ by \cite[Definition 2.5]{Der16}. Therefore, the same argument as in the proof of Theorem \ref{thm:Main} applies.
\end{rmk}

Combining Theorem \ref{thm:Main} and Proposition \ref{prop:Properties} we can strengthen the results of Odaka and Sano \cite[Theorem 1.10]{OS12}.

\begin{cor} \label{cor:Alpha} Let $X$ be a Fano variety of dimension $n$ and let $G \subset \aut(X)$ be a closed connected subgroup. If $\alpha_G(X) > \frac{n}{n + 1}$ (resp. $\alpha_G(X) \geqslant \frac{n}{n + 1}$) then $(X, -K_X)$ is $G$-equivariantly uniformly $K$-stable (resp. $G$-equivariantly $K$-semistable).
\end{cor}

\begin{exa} \label{exa:pgl} Let $X$ be a Mukai--Umemura threefold or a $V_5$ threefold. Then $\aut^0(X) \simeq \mathrm{PGL}_2(\mathbb{C})$. Taking $G = \aut^0(X)$ we have $\alpha_G(X) = \frac56$ by \cite{Don08} and \cite{CS09}. Therefore these Fano varieties are $G$-equivariantly uniformly $K$-stable by Corollary \ref{cor:Alpha}.
\end{exa}

\begin{rmk} Suppose that the space $\Val^G_X$ is empty (e.g. $X$ is a $G$-homogeneous variety). Then for every $G$-equivariant special test configuration $(\XX^s, \LL^s)$ the corresponding valuation on $X$ is trivial. Therefore we have $J^{NA}(\XX^s, \LL^s) = 0$ and therefore the test configuration $(\XX^s, \LL^s)$ is trivial by \cite[Theorem 1.3]{Der16}. Thus the $G$-equivariant $K$-stability condition is trivially satisfied. Note that the existence of K\"ahler--Einstein metrics on compact K\"ahler homogeneous manifolds was established by Matsushima in \cite[Theorem 3]{Mat72}.
\end{rmk}

\begin{prop} Let $X$ be a smooth Fano variety and suppose that $(X, -K_X)$ is $G$-equivariantly uniformly $K$-stable for a connected reductive subgroup $G \subset \aut(X)$. Then the centralizer $C_{\aut(X)}(G)$ of $G$ inside $\aut(X)$ is finite; in particular, the groups $G$ and $\aut(X)$ are semisimple.
\end{prop}

\begin{proof} By \cite[Theorem 1]{DS16} and Remark \ref{rmk:Stab} the variety $X$ admits a K\"ahler-Einstein metric, therefore the groups $\aut(X)$ and $C_{\aut(X)}(G)$ are reductive by Matsushima's theorem. The condition of $G$-equivariant uniform $K$-stability implies that every $G$-equivariant product-type test configuration is $G$-equivariantly trivial. This implies that there are no one-parameter subgroups $\mathbb{G}_m \subset C_{\aut(X)}(G)$. Since $C_{\aut(X)}(G)$ is a reductive algebraic group, its connected component of the identity is trivial, so $C_{\aut(X)}(G)$, and therefore $C(G)$, is finite. Finally, since any $\mathbb{G}_m \subset C(\aut(X))$ lies in $C_{\aut(X)}(G)$ for any $G \subset \aut(X)$ we obtain that the center of $\aut(X)$ is discrete. Using the fact that any reductive algebraic group with discrete center is semisimple (see e. g. \cite[Proposition 19.3]{Mil17}) we obtain the statement of the theorem.
\end{proof}

\begin{rmk} The above proposition shows that a K\"ahler--Einstein Fano manifold $X$ may not be uniformly $G$-equivariantly $K$-stable for any connected subgroup $G \subset \aut(X)$. For example, a smooth del Pezzo surface $X$ of degree 6 is isomorphic to a blow-up of $\mathbb{P}^2$ at three points and $\aut^0(X) \simeq (\mathbb{C}^{\times})^2$. Since the latter group is reductive, by \cite{Ti90} the surface $X$ admits a K\"ahler--Einstein metric.
\end{rmk}

\subsection{The greatest Ricci lower bound} Let now $X$ be a smooth Fano variety and let $G \subset \aut(X)$ be a connected reductive subgroup. We denote by $K_G$ the compact real form of $G$ (see e.g. \cite[Theorem 5.2.8]{OV90}); then $K_G$ is a compact Lie group acting on $X$ by automorphisms. Our Theorem \ref{thm:Main} together with \cite[Theorem 1]{DS16} shows that the inequality $\delta_G(X) > 1$ guarantees the existence of a $K_G$-invariant K\"ahler--Einstein metric on $X$. In this subsection we show that our invariant $\delta_G(X, -K_X)$ is related to metric geometry of $(X, -K_X)$ even in the case when $\delta_G(X) \leqslant 1$. Namely, we consider the following invariant: $$\beta_G(X, -K_X) = \sup\{t \in \mathbb{Q} \mid \mbox{there exists a $K_G$-invariant K\"ahler form $\omega \in c_1(X)$ with $\mathrm{Ric}(\omega) \geqslant t\omega$}\}.$$ This is a $K_G$-invariant version of {\em the greatest Ricci lower bound} $\beta(X, -K_X)$ defined by Rubinstein in \cite[Equation (32)]{Rub08}. For the case $G = \{1\}$ we have the following result.

\begin{thm}[{see \cite[Theorem 5.7]{CRZ19} and \cite[Theorem 7.5 and Corollary 7.6]{BBJ18}}] \label{thm:Rub} Let $X$ be a smooth Fano variety. Then its delta-invariant $\delta(X, -K_X)$ and the greatest Ricci lower bound $\beta(X, -K_X)$ are related by the formula $$\beta(X, -K_X) = \min\{1, \delta(X, -K_X)\}.$$
\end{thm}

We prove an analogous result in $G$-equivariant setup and establish relations between the above invariants (with and without $G$-action).

\begin{prop} \label{prop:Beta} Let $X$ be a smooth Fano variety and let $G \subset \aut(X)$ be a connected reductive subgroup. Then there is an equality \begin{equation}\beta_G(X, -K_X) = \min\{1, \delta_G(X, -K_X)\}.\label{first}\end{equation} We also have the following relation between $\delta$-invariants \begin{equation}\delta(X, -K_X) = \min\{1, \delta_G(X, -K_X)\}.\label{third}\end{equation}
\end{prop}

\begin{proof} Let us first prove equality (\ref{first}). Note that by $K_G$-invariant continuity method (see e.g. \cite{DS16}) we have an equality of the greatest Ricci lower bounds $$\beta(X, -K_X) = \beta_G(X, -K_X).$$ Therefore, we combine this with Theorem \ref{thm:Rub} and an obvious inequality $\delta(X, -K_X) \leqslant \delta_G(X, -K_X)$ to obtain $$\beta_G(X, -K_X) = \beta(X, -K_X) = \min\{1, \delta(X, -K_X)\} \leqslant \min\{1, \delta_G(X, -K_X)\}.$$

Conversely, let us choose $t \in (0; 1] \cap \mathbb{Q}$ such that $0 < t < \delta_G(X, -K_X)$. Then by \cite[Remark ~4.5]{BlJ20}, we have $\delta_G(X, -tK_X) > 1$. Therefore by Remark \ref{rmk:Rescale} the polarized variety $(X, -tK_X)$ is uniformly $G$-equivariantly $K$-stable. Then by \cite[Proposition 10]{DS16} the equation $$\mathrm{Ric}(\omega) = t\omega + \alpha$$ has a $K_G$-invariant solution for some $K_G$-invariant K\"ahler form $\alpha \in (1 - t)c_1(X)$. This shows that $$\min\{1, \delta_G(X, -K_X)\} \leqslant \beta_G(X, -K_X)$$ and proves equality (\ref{first}).

The equality (\ref{third}) follows from the previous one and Theorem \ref{thm:Rub} in the case when $\delta(X, -K_X) \leqslant 1$; otherwise, by \cite[Theorem 5.4]{BBEGZ19}, we have $G = \{1\}$ and equality (\ref{third}) then holds tautologically.
\end{proof}

\begin{rmk} We expect that an appropriate version of the Proposition \ref{prop:Beta} holds for klt Fano varieties and singular K\"ahler metrics. Also, it would be interesting to see a purely algebro-geometric proof of the equality $\delta(X) = \min\{1, \delta_G(X)\}$.
\end{rmk}

\section{Varieties with torus action} 

In this section we consider the case of a variety $X$ with an action of a torus $T = (\mathbb{G}_m)^k \subset \aut(X)$.

\subsection{Varieties with an action of a torus} Toric varieties form a well-understood class of varieties with large groups of symmetries. The formulas for the $\alpha$-invariant of a toric variety in terms of the associated fan were given in \cite{BS99}, in \cite{S05} using analytic methods, and in \cite[Lemma 5.1]{CS08}. The analogous formula for the $\delta$-invariant first appeared in \cite[Corollary 7.16]{BlJ20}. The underlying idea in these formulas is that it suffices to consider only torus-invariant divisors (or valuations) in the computation of $\alpha(X)$ and $\delta(X)$. We show that the same is true for a subtorus $T \subset \aut(X)$ of any dimension. 

\begin{prop} \label{prop:Solv}Let $(X, L)$ be a variety with an ample polarization and let $T \subset \aut(X, L)$ be a subtorus. Then the following equalities hold: $$\alpha_T(X, L) = \alpha(X, L) \qquad \mbox{and} \qquad \delta_T(X, L) = \delta(X, L).$$
\end{prop}

\begin{proof} We apply the degeneration to the initial filtration argument from \cite{BlJ20}. Let us give an outline of their argument. Attached to a valuation $v \in \Val_X$ is the filtration $\mathcal{F}_v$ on $R(X, L)$. Using a construction from \cite{KK12} we can associate to $\mathcal{F}_v$ the so-called initial filtration $\mathrm{in}(\mathcal{F}_v)$. Its crucial property is that the sequence of base ideals $$\mathfrak{b}_{\bullet}(\mathrm{in}(\mathcal{F}_v)) = \{\mathfrak{b}_{t, m}(\mathrm{in}(\mathcal{F}_v))\}, t \in \mathbb{R}, m \in \mathbb{N}$$ consists of monomial ideals and therefore is invariant under the action of $T$. Moreover, by \cite[Proposition 7.13]{BlJ20} it is possible to degenerate $\mathcal{F}_v$ to $\mathrm{in}(\mathcal{F}_v)$ in a one-parameter family in such a way that the log canonical threshold of the sequence defined by $$\lct(\mathfrak{b}_{\bullet}) = \inf_{v \in \Val_X}\frac{A_X(v)}{v(\mathfrak{b}_{\bullet})}$$ does not increase after passing to $\mathrm{in}(\mathcal{F}_v)$, that is $$\lct(\mathfrak{b}_{\bullet}(\mathrm{in}(\mathcal{F}_v))) \leqslant \lct(\mathfrak{b}_{\bullet}(\mathcal{F}_v)).$$ By \cite[Proposition 8.1]{JM12} we can associate to $\mathrm{in}(\mathcal{F}_v)$ a $T$-invariant valuation $\bar v$ computing the log canonical threshold $\lct(\mathfrak{b}_{\bullet}(\mathrm{in}(\mathcal{F}_v)))$. Such valuation will be a monomial valuation on a certain $T$-model of $X$. The valuation $\bar v$ has the property that $$\lct(\mathfrak{b}_{\bullet}(\mathrm{in}(\mathcal{F}_v)) = A_X(\bar v) \leqslant A_X(v) = \lct(\mathfrak{b}_{\bullet}(\mathcal{F}_v)$$ by \cite[Lemma 1.1]{BlJ20}. Moreover, we have $S_L(\bar v) \geqslant S_L(v)$ by the argument in \cite[Proposition 6.8]{BlJ20}. Therefore, the infimum in the definition of $\delta(X, L)$ is at least the infimum over the subspace $\Val^T_X$. The statement now follows from the definition of $\delta_T(X, L)$. The proof for $\alpha_T(X, L)$ is the same.
\end{proof}

\begin{rmk} In the case of $\alpha(X)$ this result was proved in \cite[Corollary 1.8]{OS12} for any connected solvable group using the Borel fixed point theorem. We expect Proposition \ref{prop:Solv} to hold in the case of a connected solvable group $G$.
\end{rmk}

By combining the criterion for $T$-equivariant semistability from Theorem \ref{thm:Main} and Proposition \ref{prop:Solv} we immediately obtain the following corollary.

\begin{cor} A Fano variety $(X, -K_X)$ is $K$-semistable if and only if it is $K$-semistable with respect to all $T$-equivariant test configurations.
\end{cor}

This result shows that for varieties with an action of a torus $T = (\mathbb{G}_m)^k$ it is necessary to consider additional symmetries in order to use Tian's criterion (or Theorem \ref{thm:Main}). This method was first  implemented in \cite{BS99} for symmetric toric Fano manifolds. In \cite{Su13} the method was applied to $T$-varieties of complexity one, and in \cite{CS18} for Fano threefolds from the $V_{22}$ family having automorphism groups $\mathbb{G}_m \rtimes \mathbb{Z}/2\mathbb{Z}$ (cf. \cite{DKK17} where the additional symmetries were not used).

\section{Spherical Fano varieties} A natural generalization of toric varieties is the class of spherical varieties. A variety $X$ is spherical if it has an action of a connected reductive group $G$ such that a Borel subgroup $B \subset G$ acts on $X$ with an open orbit. For general information on spherical varieties we refer to \cite{BLV86}, see also \cite{Kn91, Tim11}. Log canonical thresholds of spherical varieties were investigated in \cite{Pas16, Smi17, Del15}. An extensive study of $K$-stability of spherical Fano varieites was undertaken by Delcroix in \cite{Del16}. 

In case of a spherical $G$-variety the space of $G$-invariant valuations has a particularly nice description (see Theorem \ref{thm:Cone} below). For a Fano variety $X$, spherical under the action of a connected reductive group $G$ we give a formula for $\delta_G$ in terms of the combinatorial data defined by $X$. Moreover, we recover the combinatorial criterion for $G$-equivariant $K$-polystability of a spherical Fano variety given by Delcroix \cite[Theorem A]{Del16}. Our proof is different, though not entirely independent of the one in \cite{Del16}.

We fix the notation, mostly following \cite{Del16}. Let $X$ be a projective variety, spherical under the action of a connected reductive group $G$. Let $B \subset G$ be a Borel subgroup and let $T \subset B$ be a maximal torus. We denote by $\mathfrak{X}(T)$ the group of algebraic characters of $T$. Let also $\Phi$ be the root system of $(G, B, T)$ and $\Phi^+$ be the set of positive roots determined by $B$. We  also let $\mathfrak{N}(T)$ be the group of $1$-parameter subgroups of $T$. 

We describe the set of $G$-invariant valuations on $X$. This set depends on the open $G$-orbit $U$ only. Let us fix a Borel subgroup $B \subset G$. Associated to $B$ is the space of $B$-eigenfunctions $$M_B(U) = \{\chi \in \mathfrak{X}(B) \mid \mbox{$b \cdot f = \chi(b)f$ for all $b \in B$ and some $f \in \mathbb{C}(X)^*$}\}.$$ We denote by $N_B(U)$ the dual of $M_B(U)$; it is a free abelian group of finite rank. To every $v \in \Val^G_X$ we can associate a vector $\rho_v \in N_B(U) \otimes \mathbb{R}$ by the rule $\rho_v(\chi) = v(f)$ where $f \in \mathbb{C}(X)^*$ is a rational function as in definition of $M_B(U)$. This is a well defined map exactly because $B$ has an open orbit. The space $N_B \otimes \mathbb{R}$ is a quotient of $\mathfrak{N}(T) \otimes \mathbb{R}$; we denote by $\pi \colon \mathfrak{N}(T) \otimes \mathbb{R} \to N_B \otimes \mathbb{R}$ the quotient map. The next result gives a very simple description of the set $\Val^G_X$.

\begin{thm}[{see e.g. \cite[Corollaries 1.8 and 5.3]{Kn91}}] \label{thm:Cone} The map $\rho \colon \Val^G_X \to N_B(U) \otimes \mathbb{R}$ is injective and identifies the set $\Val^G_X$ with a polyhedral convex cone $\mathcal{V}$ in the finite-dimensional space $N_B(U) \otimes \mathbb{R}$.
\end{thm}

In \cite{Pas17} Pasquier described the restriction of the log discrepancy function $A_X$ to the space $\Val^G_X$ using the identification of the latter space with the cone $\mathcal{V}$ given by the above Theorem \ref{thm:Cone}. This description uses the so-called complete colored fan $\mathbb{F}_X$ associated to $X$, which is a finite set of finitely generated cones $\mathcal{C} \subset N_B \otimes \mathbb{R}$ with some additional data (see e.g. \cite[Definition 2.13]{Pas17} for details).

\begin{prop}[see {\cite[Theorem 2.20 and Proposition 5.2]{Pas17}}] \label{prop:Pas} Let $X$ be a $\mathbb{Q}$-Gorenstein spherical variety and let $\mathbb{F}_X$ be the complete colored fan associated to $X$. For every cone $\mathcal{C} \subset \mathbb{F}_X$ there exists a function $h_{\mathcal{C}} \colon N_B \otimes \mathbb{R} \to \mathbb{R}$ such that its restriction to the cone $\mathcal{C} \in \mathbb{F}_X$ is linear and for every $v \in \mathcal{V} \cap \mathcal{C}$ we have $A_X(v) = h_{\mathcal{C}}(v)$. If $v = \mathrm{ord}_{D}$ corresponds to a primitive generator of an edge of $\mathcal{C}$ then $h_{\mathcal{C}}(v) = a_{D}$ with $a_D = 1$ or $a_D \geqslant 2$.
\end{prop}

The coefficients $a_D$ can be expressed via the spherical roots and coroots of the triple $(G, B, T)$ (see \cite[Theorem 2.20]{Pas17} for details). Note also that by \cite[Proposition 5.6]{Pas17} a $\mathbb{Q}$-Gorenstein spherical variety $X$ always has klt singularities.

In \cite{Del16} Delcroix investigated $G$-equivariant $K$-stability of spherical Fano varieties. He gave a construction of $G$-equivariant special test configurations $(\XX, \LL)$ corresponding to vectors in the valuation cone $\mathcal{V}$. Moreover, he computed the Donaldson--Futaki invariants of these test configurations in terms of the moment polytope $\Delta^+ \subset \mathfrak{X}(T)$ (see \cite[Definition 3.14]{Del16}) and the so-called  Duistermaat--Heckman measure $\mathrm{DH}$ on $\Delta^+$ (see \cite[Theorem 4.5]{Del16}). The moment polytope determines a subsystem $\Phi_L \subset \Phi$; we denote by $2\rho_Q$ the sum of elements in $\Phi^+ \setminus \Phi_L$. 

\begin{thm}[{\cite[Theorems B and C]{Del16}}] \label{thm:Delcroix} Let $(X, -K_X)$ be a spherical Fano variety. To every rational vector $v \in \mathcal{V}$ corresponds a $G$-equivariant test configuration $(\XX^v, \LL^v)$ for $(X, -K_X)$ with irreducible central fiber $\XX^v_0$. Moreover, there exists $m \in \mathbb{N}$ such that the test configuration constructed from $mv$ is special and, conversely, every $G$-equivariant special test configuration for $(X, -K_X)$ can be constructed in this way. The Donaldson--Futaki invariant of a special test configuration corresponding to $v$ is given by the formula $$\mathrm{DF}(\XX^v, \LL^v) = V \cdot \langle2\rho_Q - \mathrm{bar}_{DH}(\Delta^+), \pi^{-1}(v) \rangle.$$ Here $\mathrm{bar}_{DH}(\Delta^+)$ is the barycenter of the moment polytope with respect to the Duistermaat--Heckman measure $\mathrm{DH}$ and $V$ is a universal constant depending on $\Delta^+$ and $\mathrm{DH}$ only. The above expression does not depend on the choice of a preimage $\pi^{-1}(v)$ of $v$.
\end{thm}

This computation allows to express the $\delta_G$-invariant of $X$ in terms of the combinatorial data associated to $X$ and $B \subset G$ and prove Proposition \ref{prop:Form1}. Note that the formula below does not depend on the choices of a Borel subgroup $B$ and a maximal torus $T \subset B$.

\begin{prop} \label{prop:Form} Let $X$ be a Fano variety which is spherical under the action of $G$; let $B \subset G$ be a Borel subgroup and $T \subset B$ a maximal torus. In the above notation, the following formula holds for the $\delta_G$-invariant of $X$: $$\delta_G(X) = \min_{\substack{\mathrm{ord}_{D_i} \in \mathcal{V} \cap \mathcal{C} \\ \mathcal{C} \in \mathbb{F}_X}}\frac{a_{D_i}}{a_{D_i} - V \cdot \langle2\rho_Q - \mathrm{bar}_{DH}(\Delta^+), \pi^{-1}(\mathrm{ord}_{D_i}) \rangle}.$$ Here the minimum is taken over a finite set $\mathrm{ord}_{D_1}, \ldots, \mathrm{ord}_{D_N}$ of divisorial valuations corresponding to primitive generators of edges in $\mathcal{V} \cap \mathbb{F}_X$.
\end{prop}

\begin{proof} By Theorem \ref{thm:Cone} we can identify $\Val^G_X$ with the cone $\mathcal{V}$. Since the functions $A_X$ and $S_{-K_X}$ are homogeneous of order $1$ by Propositions \ref{prop:Extend} and \ref{prop:Extend2}, their ratio depends only on the line generated by $v \in \mathcal{V}$. Thus for any $v \in \mathcal{V}$ we can consider the special test configuration $(\XX^v, \LL^v)$ from Theorem \ref{thm:Delcroix}. Using Proposition \ref{prop:SpecialTC} we find $$\mathrm{DF}(\XX^v, \LL^v) = V \cdot \langle2\rho_Q - \mathrm{bar}_{DH}(\Delta^+), \pi^{-1}(v) \rangle = A_X(v) - S_L(v).$$ By Proposition \ref{prop:Pas} the log discrepancy function identifies with the piecewise linear function on the cones $\mathcal{C} \in \mathbb{F}_X$. Therefore the $\delta_G$-invariant is the minimum over the cones in $\mathcal{V} \cap \mathbb{F}_X$ of two piecewise linear functions. Thus the minimum is attained on one of the primitive generators of edges in $\mathcal{C} \cap \mathcal{V}, \mathcal{C} \in \mathbb{F}_X$ corresponding to one of the valuations $\mathrm{ord}_{D_1}, \ldots, \mathrm{ord}_{D_N}$. By Proposition \ref{prop:Pas} we have $A_X(\mathrm{ord}_{D_i}) = a_{D_i}$ where $a_{D_i} = 1$ or $a_{D_i} \geqslant 2$. The formula now follows.
\end{proof}

\begin{rmk} It should be possible to prove the formula for $S_{-K_X}(v)$ for $v \in \mathcal{V}$ directly, not relying on the computation of \cite{Del16}.
\end{rmk}

\begin{exa} In the case when $X$ is a toric variety and $G = (\mathbb{G}_m)^n$ the polytope $\Delta^+$ is the usual polytope $\Delta$ associated to $X$. The cone of $G$-invariant valuations is generated by the valuations $v_{D_i}$ corresponding to the torus-invariant divisors $D_1, \ldots, D_d$. The log discrepancy function is linear on the cones of the fan of $X$. Therefore, the infumum is attained on one of the valuations $v_{D_i}$ (see \cite[Corollary ~7.4]{BlJ20}). Thus, by Proposition \ref{prop:Solv} we recover the formula from \cite[Corollary 7.16]{BlJ20}: $$\delta(X) = \delta_G(X) = \min_{1 \leqslant i \leqslant d} \frac{1}{1 + \langle \mathrm{bar}(\Delta), v_i\rangle}.$$
\end{exa}

\begin{exa} More generally, let $X$ be a horospherical variety, that is, a $G$-equivariant compactification of a homogeneous space $G/H$ where $H$ contains the unipotent radical of a Borel subgroup of $G$. These varieties admit a characterization in terms of the valuation cone $\mathcal{V}$. Namely, by \cite[Corollary 6.2]{Kn91} a spherical variety $X$ is horospherical if and only if $\mathcal{V} = N_B(U) \otimes \mathbb{R}$. Assume that the valuation cone of $X$ is non-trivial. Then by Proposition \ref{prop:Form} for a horospherical Fano variety $X$ we have $\delta_G(X) \leqslant 1$. Moreover, if $X$ is smooth then by \cite[Corollary ~5.7]{Del16} the condition $2\rho_Q = \mathrm{bar}_{DH}(\Delta^+)$ (equivalent to $\delta_G(X) = 1$) is a criterion for existence of a K\"ahler--Einstein metric on $X$.
\end{exa}

On the other hand, we have some obvious examples of horospherical Fano varieties with $\delta_G(X) > 1$.

\begin{exa} Let $X$ be a spherical homogeneous space under the action of $G$, for example a Grassmannian ~$\mathrm{Gr}(k, n)$ with $G = \mathrm{GL}_n(\mathbb{C})$. Then, as in Remark \ref{rmk:Homog}, we obtain that the space $\dival^G_X$ is empty and therefore $\delta_G(X) = \infty$.
\end{exa}

To produce more examples of spherical Fano varieties $X$ satisfying $\delta_G(X) > 1$ we need to study the group $\aut_G(X)$ of $G$-equivariant automorphisms of $X$. If the spherical variety $X$ is an equivariant compactification of the quotient $G/H$ then this group can be described by $$\aut^0_G(X) \simeq (N_G(H)/H)^0$$ where $N_G(H)$ is the normalizer of $H$ inside $G$ (see e. g. \cite[Section 3.1.3]{Del16} and references therein). It turns out that the dimension of $\aut_G(X)$ can be recovered from the valuation cone $\mathcal{V}$, namely, from its {\em linear part} $\mathcal{V} \cap (-\mathcal{V})$.

\begin{prop}[{see e. g. \cite[Theorem 6.1]{Kn91}}]\label{prop:LinPart} Let $X$ be a spherical variety and let $\aut_G(X)$ be the group of $G$-equivariant automorphisms of $X$. Then the connected component $\aut_G^0(X)$ is a central torus in $G$. Moreover, the dimension of $\aut_G(X)$ is equal to the dimension of the linear part $\mathcal{V} \cap (-\mathcal{V})$ of the valuation cone.
\end{prop}

Using Theorem \ref{thm:Main}, we can deduce the statement of \cite[Theorem A]{Del16} (in the non-twisted case) and establish an analogous criterion for $G$-equivariant uniform $K$-stability.

\begin{prop}[see {\cite[Theorem A]{Del16}}] Let $X$ be a $G$-spherical Fano variety. Then $(X, -K_X)$ is $G$-equivariantly $K$-polystable if and only if the vector $$\mathrm{bar}_{DH}(\Delta^+) - 2\rho_Q \in \mathfrak{X}(T) \otimes \mathbb{R}$$ lies in the relative interior of the dual cone to the closure of $\pi^{-1}(-\mathcal{V})$. Analogously, the variety $(X, -K_X)$ is $G$-equivariantly uniformly $K$-stable if and only if the vector $2\rho_Q - \mathrm{bar}_{DH}(\Delta^+)$ belongs to the relative interior of the dual cone to the closure of $\pi^{-1}(-\mathcal{V})$ and, in addition, the linear part of the valuation cone $\mathcal{V}$ is trivial.
\end{prop}

\begin{proof} By Theorem \ref{thm:Main} it follows that $(X, -K_X)$ is $G$-equivariantly $K$-semistable iff $\delta_G(X) \geqslant 1$. By our formula for $\delta_G(X)$ from Proposition \ref{prop:Form} we obtain that the condition $\delta_G(X) \geqslant 1$ is equivalent to $$\left\langle2\rho_Q - \mathrm{bar}_{DH}(\Delta^+), \pi^{-1}(D_i) \right\rangle \geqslant 0$$ for all $D_i \in \mathcal{V}$. This means that the vector $\mathrm{bar}_{DH}(\Delta^+) - 2\rho_Q$ belongs to the dual cone to $\pi^{-1}(-\mathcal{V})$. Moreover, a vector $\xi \in \mathfrak{N}(T) \otimes \mathbb{Q}$ corresponds (up to a multiple) to a product-type $G$-equivariant test configuration if and only if $\xi$ lies on one of the hyperplanes defining the dual cone to $\overline{\pi^{-1}(-\mathcal{V})}$. The latter happens precisely when $\xi$ projects to the linear part of $\mathcal{V}$ (see \cite[the proof of Theorem 5.3]{Del16}). This implies the first statement of the proposition.

If $X$ is $G$-equivariantly uniformly $K$-stable then again by Theorem \ref{thm:Main} and Proposition \ref{prop:Form} we have 
$$\langle2\rho_Q - \mathrm{bar}_{DH}(\Delta^+), \pi^{-1}(v)\rangle > 0$$ for any $v \in \mathcal{V}\setminus\{0\}$. Thus the vector $2\rho_Q - \mathrm{bar}_{DH}(\Delta^+)$ belongs to the dual cone to the closure of $\pi^{-1}(-\mathcal{V})$. Moreover, by linearity of the $\mathrm{DF}$ functional, the linear part $\mathcal{V} \cap -\mathcal{V}$ has to be trivial. Conversely, if the condition $$\langle2\rho_Q - \mathrm{bar}_{DH}(\Delta^+), \pi^{-1}(v)\rangle > 0$$ is satisfied for any $v \in \mathcal{V}\setminus\{0\}$, then the linear part $\mathcal{V} \cap (-\mathcal{V})$ is equal to $\{0\}$. By Proposition \ref{prop:LinPart} we obtain that $\aut_G(X)$ is finite; hence, any product-type $G$-equivariant test-configuration is trivial. Therefore by Theorem \ref{thm:Main} we obtain that $(X, -K_X)$ is $G$-equivariantly $K$-stable; since the valuation cone is polyhedral, $(X, -K_X)$ is also uniformly $G$-equivariantly $K$-stable.
\end{proof}

\begin{exa} Let $G/H$ be a symmetric spherical homogeneous space with $G$ semisimple. By the result of De Concini and Procesi \cite[Theorem 3.1]{dCP83} there exists a simple compactification $X$ of $G/H$, i. e. a $G$-equivariant compactification of $G/H$ such that $X$ contains exactly one closed $G$-orbit. In fact, it follows from Proposition \ref{prop:LinPart} that existence of a simple compactification is equivalent to finiteness of $N_G(H)/H$. Therefore, if $X$ is a smooth symmetric spherical Fano variety then existence of a K\"ahler--Einstein metric on $X$ is equivalent to the condition $\delta_G(X) > 1$. Symmetric Fano varieties obtained as blow-ups of wonderful compactifications were described by Ruzzi \cite[Theorem B]{Ruz12}, see also \cite[Section 5.4.2]{Del16} and examples therein.
\end{exa}

\section{Finite automorphism groups}

In this section we show that in case of a variety $X$ with a finite group action, we can adapt Definition ~\ref{defi:Divisors} to $G$-equivariant setting. Also we compare the invariant $\delta_G$ defined below to the $\delta$-invariant of the quotient $Y = X/G$ with the orbifold pair structure. We expect that there is an analogue of our main Theorem \ref{thm:Main} for the case of a finite group $G$. 

Being motivated by the results of \cite{Su13}, we would like to consider the case of a complexity-one Fano $T$-variety $X$. Let $G$ be a subgroup of the normalizer $N_{\aut(X)}(T)$. Then $G$ acts on the Chow quotient $Y \sim \mathbb{P}^1$ of $X$ by the $T$-action. Let $B = \sum_i(1 - \frac{1}{m_i})[p_i]$ be the $G$-invariant branch $\mathbb{Q}$-divisor on $Y$ encoding the multiple fibers of the quotient fibration. Then, inspired by \cite[Theorem 1.2]{Su13}, we can obtain conditions on $T \rtimes G$-equivariant $K$-stability of $X$ in terms of $\delta_G(Y, B)$ (defined below).

\begin{defi} \label{defi:Finite} Let $(X, \Delta; L)$ be a pair with a big $\mathbb{Q}$-divisor $L$. Let $G$ be a finite subgroup of the automorphism group $\mathrm{Aut}(X, \Delta; L)$. Then we define the $\alpha_{G, m}$-invariant for the pair $(X, \Delta; L)$ by $$\alpha_{G, m}(X, \Delta; L) = \inf\{\lct(X, \Delta; D) \mid \mbox{$D \sim mL$ and $D$ is effective and $G$-invariant}\}.$$ Analogously, let us define $$\delta_{G, m}(X, \Delta; L) = \inf\{\lct(X, \Delta; D) \mid \mbox{$D \sim mL$ is a $G$-orbit of a $\QQ$-divisor of $m$-basis type}\}.$$ Here by a $G$-orbit of an $m$-basis type divisor we mean a $\mathbb{Q}$-divisor $D \in \PP(H^0(X, -mL))$ which belongs to the image of the open set of $m$-basis type divisors under a surjective averaging map $$\PP(H^0(X, -mL)) \to \PP(H^0(X, -mL)^G).$$ We define $$\alpha_{G}(X, \Delta; L) = \inf_{m \in \mathbb{N}}\alpha_{G, m}(X, \Delta; L) \qquad \mbox{and} \qquad \delta_G(X, \Delta; L) = \lim_{m \to \infty}\delta_{G, m}(X, \Delta; L)$$ for $m \in \mathbb{N}$ sufficiently large and divisible.
\end{defi}

\begin{rmk} \label{rmk:Basis} In the above definition we consider $G$-invariant divisors with several irreducible components (permuted by the action of $G$). Note that by lower semi-continuity of log canonical thresholds \cite[Lemma 7.8]{KP17} the infima from the definition of $\delta_{G, m}(X, \Delta, L)$ are attained on open subsets of $\PP(H^0(X, -mL)^G$. Moreover, for a Fano variety $(X, -K_X)$ and a finite group $G \subset \aut(X)$ we have an isomorphism $H^0(X, -m|G|K_X) \sim H^0(Y, -m(K_Y + B))$ where $Y = X / G$ is a quotient variety and $B$ is the branch divisor. Hence the above infimum can be computed using preimages of $m$-basis type divisors from $Y$. In particular, this shows the equivalence of Definition \ref{defi:Finite} and \cite[Definition 2.5]{LZ20}.
\end{rmk}

\begin{exa} \label{exa:Line} Consider the case $X = \mathbb{P}^1$ and $G = \mathbb{Z}/m\mathbb{Z}$. Then any action of $G$ on $\mathbb{P}^1$ has two fixed points, so the minimal length of a $G$-orbit on $X$ is ~1. Thus for any $G$-invariant divisor $D$ we have $$\sup\{t \in \mathbb{Q} \mid -K_{\PP^1} - tD \geqslant 0\} \leqslant 2 \qquad  \mbox{and} \qquad S(D) \leqslant \frac12\int^2_0(2 - t)dt = 1.$$ Thus we get $$\alpha_G(\mathbb{P}^1) = \frac12\qquad \mbox{and} \qquad \delta_G(\mathbb{P}^1) = 1.$$ More generally, if $G$ is a finite group acting faithfully on $\mathbb{P}^1$ then by the same argument we have $$\delta_G(\mathbb{P}^1) = 2\alpha_G(\mathbb{P}^1)$$ which is equal to the minimal length of an orbit under a given $G$-action on $\mathbb{P}^1$ (see \cite[Example 2]{CPS18}).
\end{exa}

We now prove the ramification formula (Proposition \ref{prop:Ramif1} from the introduction).

\begin{prop} \label{prop:Ramif} Let $X$ be a normal and $\mathbb{Q}$-Gorenstein variety with at most log terminal singularities and $-K_X$ big. Consider a finite subgroup $G \subset \mathrm{Aut}(X)$. Let us denote by $Y = X / G$ the quotient variety, by $\pi \colon X \to Y$ the quotient map and by $B = \sum(1 - \frac{1}{m_i})B_i$ the branch $\mathbb{Q}$-divisor on $Y$. Then the following ramification formulas hold for $\alpha_G(X)$ and for $\delta_G(X)$: $$\alpha_G(X) = \alpha(Y, B) \qquad \mbox{and} \qquad \delta_G(X) = \delta(Y, B).$$
\end{prop}

\begin{proof} For a $\mathbb{Q}$-divisor $D_Y$ on $Y$ its pullback $\pi^*D_Y$ is a $G$-invariant $\QQ$-divisor on $X$. We also define the crepant preimage $D_X$ of $D_Y$ via $\pi$ by the formula $$K_X + D_X \Qlin \pi^*(K_Y + B + D_Y).$$ By the canonical bundle formula for a finite morphism we have $-K_X \Qlin -\pi^*(K_Y + B)$, in particular, the pair $(Y, B)$ is klt. Consider a $\QQ$-divisor $D_Y \Qlin -(K_Y + B)$. Then its pullback $\pi^*D_Y$ is a $G$-invariant $\QQ$-divisor $\QQ$-linearly equivalent to $-K_X$. Conversely, every $G$-invariant $\QQ$-divisor $D \sim_{\mathbb{Q}} -K_X$ is a pullback of a $\QQ$-divisor $D_Y = \pi(D)$ on $Y$. In particular, $D_Y \sim_{\mathbb{Q}} -(K_Y + B)$ is of $m$-basis type if and only if $\pi^*D_Y$ is of $m|G|$-basis type. Moreover by \cite[Proposition 3.16]{Kol97} the pair $(Y, D_Y)$ is klt if and only if the pair $(X, D_X)$ is klt. Therefore by Remark \ref{rmk:Basis} we can write \begin{align*} \delta_m(Y, B) &= \inf\{\lct(Y, B; D_Y) \mid \mbox{$D_Y \Qlin -(K_Y + B)$ is of $m$-basis type}\} &&\\ &= \inf\{\lct(X; D_X) \mid \mbox{$D_X$ is the crepant preimage of $D_Y$ as above} \} && \\ &=\inf\{\lct(X; \pi^*D_Y) \mid \mbox{$\pi^*D_Y \Qlin -K_X$ and $\pi^*D_Y$ is $G$-invariant of $m|G|$-basis type}\} = \delta_{G, m|G|}(X)&&\end{align*} for all $m \in \mathbb{N}$ and thus $\delta_G(X) = \delta(Y, B)$. Analogously, for the $\alpha$-invariant we have \begin{align*} \alpha(Y, B) &= \inf\{\lct(Y, B; D_Y) \mid D_Y \Qlin -(K_Y + B) \mbox{ and $D_Y$ is an effective $\mathbb{Q}$-divisor}\}  &&\\ &= \inf\{\lct(X; D) \mid D \Qlin -K_X \mbox{ and $D$ is a $G$-invariant effective $\mathbb{Q}$-divisor}\} = \alpha_G(X)&&\end{align*} as desired.
\end{proof}

\begin{exa} In Example \ref{exa:Line} consider the quotient map $\pi \colon \mathbb{P}^1 \to \mathbb{P}^1$ by the cyclic group $G = \mathbb{Z}/m\mathbb{Z}$. The branch divisor is equal to $$B = \frac{m-1}{m}(p_1 + p_2).$$ To compute $\alpha(\mathbb{P}^1, B)$ we note that for any effective $D \sim_{\mathbb{Q}} -K_{\mathbb{P}^1}$ and for any point $p \in \mathbb{P}^1$ we have $$\lct_p(\mathbb{P}^1, D) \geqslant 1/2$$ since $\alpha(\mathbb{P}^1) = 1/2$. Therefore for any $p$ and for any $D$ we have $\mathrm{mult}_p(D) \leqslant 2$. Thus for any effective $E \sim_{\mathbb{Q}} -(K_{\mathbb{P}^1} + B)$ we obtain that $\mathrm{mult}_p(E) \leqslant 2/n$. This implies that $$\lct_p(\mathbb{P}^1, B; E) \geqslant 1/2$$ and moreover the value $1/2$ is attained; so that $\alpha(\mathbb{P}^1, B) = 1/2$. The same computation together with the fact that $\delta(\mathbb{P}^1) = 1$ shows that $\delta(\mathbb{P}^1, B) = 1$. 
\end{exa}

\begin{rmk} We expect that there is a definition of $\delta_{T \rtimes G}(X)$ such that an estimate analogous to \cite[Theorem 1.2]{Su13} holds.
\end{rmk}

\flushleft{Aleksei Golota \\
National Research University Higher School of Economics, Russian Federation \\
Laboratory of Mirror Symmetry, NRU HSE, 6 Usacheva str.,Moscow, Russia, 119048 \\
\texttt{golota.g.a.s@mail.ru}}

\end{document}